\documentclass[12pt, reqno]{amsart}
\usepackage{amsmath, amsthm, amscd, amsfonts, amssymb, graphicx, color}
\usepackage[bookmarksnumbered, colorlinks, plainpages]{hyperref}
\input{mathrsfs.sty}

\newtheorem{theorem}{Theorem}[section]
\newtheorem{lemma}[theorem]{Lemma}
\newtheorem{proposition}[theorem]{Proposition}
\newtheorem{corollary}[theorem]{Corollary}
\theoremstyle{definition}
\newtheorem{definition}[theorem]{Definition}
\newtheorem{example}[theorem]{Example}

\theoremstyle{remark}

\begin{document}

\title[ $C^*$-convexity of norm unit balls ]{ $C^*$-convexity of norm unit balls}

\author[ M. Kian]{ Mohsen Kian}

\address{Mohsen Kian:  Department of Mathematics,  Faculty of Basic Sciences,  University of Bojnord, P.O. Box 1339, Bojnord 94531,  Iran.}
\address{School of Mathematics, Institute for Research in Fundamental Sciences (IPM), P.O. Box: 19395-5746, Tehran, Iran.}
\email{ kian@ub.ac.ir and kian@member.ams.org}



\begin{abstract}
We determine those norms on $\mathbb{B}(\mathcal{H})$ whose unit ball is $C^*$-convex. We call them $M$-norms and investigate their dual norms, say $L$-norms. We show that ``the class of $L$-norms  greater than a given norm   enjoys a  minimum element" and ``the class of $M$-norms  less than a given norm  enjoys a  maximum element". These minimum and maximum elements will be determined in some cases.  Finally, we give a constructive result to obtain   $M$-norms  and   $L$-norms on $\mathbb{B}(\mathcal{H})$.

\end{abstract} \maketitle

\section{Introduction}
Throughout this paper assume that $\mathbb{B}(\mathcal{H})$ is
 the $C^*$-algebra of all bounded linear operators on a Hilbert space $\mathcal{H}$ and $I$ denotes the identity operator on $\mathcal{H}$. If $dim \mathcal{H}=n$, then we identify $\mathbb{B}(\mathcal{H})$ with $\mathbb{M}_n$, the $C^*$-algebra of all $n\times n$  matrices with complex entries. We mean by $\mathbb{L}^1(\mathcal{H})$ the $*$-algebra of all trace class operators on $\mathcal{H}$.

  A subset $\mathscr{K}$ of $\mathbb{B}(\mathcal{H})$ is called $C^*$-convex if $A_1,\ldots,A_k\in\mathscr{K}$ and
  $C_1,\ldots,C_k\in\mathbb{B}(\mathcal{H})$ with $\sum_{i=1}^{k}C_i^*C_i=I$ implies that $\sum_{i=1}^{k}C_i^*A_iC_i\in\mathscr{K}$.  This kind of convexity has been introduced  by Loebl and Paulsen \cite{Loebl} as a non-commutative generalization of the linear convexity.    For example the sets
$\{T\in\mathbb{B}(\mathcal{H});\,0\leq T\leq I\}$ and  $\{T\in\mathbb{B}(\mathcal{H});\,\|T\|\leq M\}$
are $C^*$-convex.
 It is evident that the $C^*$-convexity of a set $\mathscr{K}$ in $\mathbb{B}(\mathcal{H})$  implies its  convexity in the usual sense. But the converse is not true in general.  Various   examples and some basic properties of  $C^*$-convex sets were presented in \cite{Loebl}.

In recent decades, many operator algebraists paid their attention to extend various concepts to non-commutative cases. Regarding these works, the notion of $C^*$-convexity in $C^*$-algebras has been established as a non-commutative generalization of the linear convexity in linear spaces.    Some essential results of convexity theory have been  generalized in \cite{EW} to $C^*$-convexity. In particular, a version of the so-called Hahn--Banach theorem was presented. The operator extension of extreme points, the $C^*$-extreme points have also been introduced and studied, see \cite{F,  FM, FZ, Hop, Loebl, M}. Moreover,  Magajna  \cite{Magajna} established  the notion of $C^*$-convexity for operator modules and proved some separation theorems. We refer the reader to  \cite{Morenz,Magajna2, WW} for further results concerning $C^*$-convexity.

    The main aim of the present work is to determine those norms whose  unit ball is $C^*$-convex (and their dual norms). We call them $M$-norms.

In Section 2 we introduce $M$-norms and $L$-norms and give their properties and connections between.

In Section 3, we will show that ``the class of $L$-norms which are greater than an arbitrary norm   $\|\cdot\|$, enjoys a  minimum element" and ``the class of $M$-norms which are less than an arbitrary norm   $\|\cdot\|$, possesses  a  minimum element".  We determine this minimum and maximum elements in some cases. For example, we will show that the trace norm $\|\cdot\|_1$  is the minimum element in the class of $L$-norms
greater than the operator norm   $\|\cdot\|_{\infty}$.

In Section 4, we give a  constructive result to obtain $M$-norms.
\section{$M$-type and $L$-type Norms}
Let $\|\cdot\|$ be a  norm on $\mathbb{B}(\mathcal{H})$ such that its unit ball  $\mathscr{B}_1=\{A\in\mathbb{B}(\mathcal{H}); \ \|A\|\leq 1\}$ is $C^*$-convex. Then
$$\left\|\sum_{i=1}^{k}C_i^*A_iC_i\right\|\leq 1 \qquad \left( \|A_i\|\leq 1, \qquad \sum_{i=1}^{k}C_i^*C_i=I \right),$$
which further implies that
  \begin{align}\label{c1}
    \left\|\sum_{i=1}^{k}C_i^*X_iC_i\right\|\leq\max_{1\leq i\leq  k}\|X_i\| \qquad\left(   X_i\in\mathbb{B}(\mathcal{H}), \qquad\ \sum_{i=1}^{k}C_i^*C_i=I\right),
  \end{align}
since $\left(\max_{1\leq j\leq k}\|X_j\|\right)^{-1}X_i\in\mathscr{B}_1$ for all $i=1,\ldots, k$. On the other hand, if the equation \eqref{c1} holds, then for all $X_i\in\mathscr{B}_1$ and $C_i\in\mathbb{B}(\mathcal{H})$  with $\sum_{i=1}^{k}C_i^*C_i=I$ we obtain $\left\|\sum_{i=1}^{k}C_i^*X_iC_i\right\|\leq 1$ and so  $\mathscr{B}_1$ is $C^*$-convex.  Let us  give a name to such norms.
\begin{definition}
We say that a norm $\|\cdot\|$ on $\mathbb{B}(\mathcal{H})$ is  a $M$-norm (is of $M$-type) if the equation \eqref{c1} holds true.
  \end{definition}
The motivation  of the next definition will be revealed soon.
\begin{definition}
 We say that a norm $\|\cdot\|$ on $\mathbb{B}(\mathcal{H})$ is a $L$-norm (is of $L$-type) if
   \begin{align}\label{cd}
  \sum_{i=1}^{k}\|C_iXC_i^*\|\leq \|X\| \qquad   \left( X\in\mathbb{B}(\mathcal{H}), \quad  \sum_{i=1}^{k}C_i^*C_i=I\right).
  \end{align}
\end{definition}
The discussion above yields that:
 \begin{lemma}\label{lem1}
  Let $\|\cdot \|$ be a norm on $\mathbb{B}(\mathcal{H})$. The unit ball of $\|\cdot \|$ is $C^*$-convex if and only if   $\|\cdot\|$ is  a  $M$-norm.
\end{lemma}
\begin{example}\label{spn}
  The operator norm (spectral norm) $\|\cdot\|_{\infty}$  is  a $M$-norm.
 It is well-known that (see for example \cite[Proposition 1.3.1]{Bhatia})
  \begin{align}\label{bh1}
  \|A\|_{\infty}\leq 1 \ \Longleftrightarrow \ \left[
\begin{array}{cc}
 I &  A\\
   A^* & I
\end{array}\right]\geq0.
\end{align}
Let $C_i\in\mathbb{B}(\mathcal{H})$ with $\sum_{i=1}^{k}C_i^*C_i=I$. If $\|A_i\|_{\infty}\leq 1$\ $(i=1,\ldots,k)$, then
\begin{align*}
 \ \left[
\begin{array}{cc}
 I &  \sum_{i=1}^{k}C_i^*A_iC_i\\
   \sum_{i=1}^{k}C_i^*A_i^*C_i & I
\end{array}\right]=\sum_{i=1}^{k}{\rm diag}(C_i,C_i)^*\left[
\begin{array}{cc}
 I &  A_i\\
  A_i^* & I
\end{array}\right]{\rm diag}(C_i,C_i)\geq 0.
\end{align*}
Hence, $\left\|\sum_{i=1}^{k}C_i^*A_iC_i\right\|_{\infty}\leq 1$ and so, the unit ball of $\|\cdot\|_{\infty}$ is $C^*$-convex. Lemma \ref{lem1} now implies that $\|\cdot\|_{\infty}$ is a $M$-norm.
 \end{example}

The dual space of $\mathbb{M}_n$ is identified with $\mathbb{M}_n$ itself under the duality coupling
$$\langle Y|X\rangle={\rm Tr} \left(Y^*X\right),\quad X,Y\in\mathbb{M}_n.$$
Here, it should be noticed that
\begin{align}\label{tr1}
\langle Y|Z^*XZ\rangle=\langle ZYZ^*|X\rangle\  \ \   \ ( X,Y,Z\in\mathbb{M}_n).
\end{align}
The dual norm $\|\cdot\|_*$ of a norm $\|\cdot\|$ is defined as
\begin{align}\label{dn}
  \|Y\|_*=\sup\left\{|\langle Y|X\rangle|;\quad \|X\|\leq 1\right\}, \ \ \ \mbox{for all}\  \ Y\in\mathbb{M}_n.
\end{align}
The next theorem provides a condition under which the unit ball of the dual norm $\|\cdot\|_*$  of a norm $\|\cdot\|$ on $\mathbb{M}_n$ is $C^*$-convex.
\begin{theorem}
  Let $\|\cdot\|_*$  be the dual norm of a norm $\|\cdot\|$ on $\mathbb{M}_n$. The unit   ball of  $\|\cdot\|_*$ is $C^*$-convex if and only if $\|\cdot\|$ is a $L$-norm.
 \end{theorem}
\begin{proof}
  The unit ball of  $\|\cdot\|_*$ is the set
 $$\mathscr{B}^*_1=\left\{Y\in\mathbb{M}_n;\ \ |{\rm Tr} \left(Y^*X\right)|\leq \|X\|, \ \ \mbox{for all}\  X\in\mathbb{M}_n\right\} .$$
 Assume that $\mathscr{B}^*_1$ is $C^*$-convex, $X\in\mathbb{M}_n$ and $\sum_{i=1}^{k}C_i^*C_i=I$. For every $i=1,\ldots,k$, the definition \eqref{dn} guarantees the existence of $Y_i\in\mathbb{M}_n$ such that $\|Y_i\|_*=1$ and $ {\rm Tr}\left(Y_i^*C_iXC_i^*\right)=\|C_iXC_i^*\|$. It follows from the $C^*$-convexity of $\mathscr{B}^*_1$ that $\sum_{i=1}^{k}C_i^*Y_iC_i\in\mathscr{B}^*_1$. Therefore,
 \begin{align*}
\sum_{i=1}^{k}\|C_iXC_i^*\|&=\left|\sum_{i=1}^{k}{\rm Tr}\left(Y_i^*C_iXC_i^*\right)\right|\\
&=\left|\sum_{i=1}^{k}{\rm Tr}\left(C_i^*Y_i^*C_iX\right)\right|\qquad (\mbox{by \eqref{tr1}})\\
&=\left|{\rm Tr}\left(\left(\sum_{i=1}^{k}C_i^*Y_iC_i\right)^*X\right)\right|\\
&\leq \|X\|,
 \end{align*}
 i.e., the condition \eqref{cd} is valid and so $\|\cdot\|$ is a  $L$-norm.

 For the converse, assume that \eqref{cd} holds. Assume that $Y_1,\ldots,Y_k\in\mathscr{B}^*_1$ and $C_1,\ldots,C_k\in\mathbb{M}_n$ with $\sum_{i=1}^{k}C_i^*C_i=I$. It follows
 \begin{align}\label{d1}
\left|{\rm Tr}\left(Y_i^*Z\right)\right|\leq \|Z\|\qquad  (  Z\in\mathbb{M}_n,\ \ i=1,\ldots,k).
 \end{align}
  Hence,
 If $X\in\mathbb{M}_n$, then
 \begin{align*}
   \left|{\rm Tr}\left(\left(\sum_{i=1}^{k}C_i^*Y_iC_i\right)^*X\right)\right|&=
   \left|\sum_{i=1}^{k} {\rm Tr}\left(Y_i^*\left(C_iXC_i^*\right)\right)\right|\qquad (\mbox{by \eqref{tr1}})\\
   &\leq \sum_{i=1}^{k}\left| {\rm Tr}\left(Y_i^*\left(C_iXC_i^*\right)\right)\right|\\
   &\leq \sum_{i=1}^{k}\left\|C_iXC_i^*\right\|\qquad (\mbox{by \eqref{d1}})\\
   & \leq \|X\|\qquad \qquad (\mbox{by \eqref{cd}}).
 \end{align*}
 Therefore, $\sum_{i=1}^{k}C_i^*Y_iC_i\in\mathscr{B}^*_1$ and so $\mathscr{B}^*_1$ is $C^*$-convex.
 \end{proof}
The connection between $M$-norms and $L$-norms is stated below.
\begin{lemma}\label{lem2}
 A norm  $\|\cdot\|$ on $\mathbb{M}_n$ is a $M$-norm (resp. $L$-norm) if and only if its dual norm $\|\cdot\|_*$ is a $L$-norm (resp. $M$-norm).
\end{lemma}
\begin{proof}
  Suppose that  $\|\cdot\|$ is a $M$-norm. Take $B\in\mathbb{M}_n$. Let $C_i\in\mathbb{M}_n$\ $(i=1,\ldots,k)$ with $\sum_{i=1}^{k}C_i^*C_i=I$. For every $i=1,\ldots,k$, it follows from \eqref{dn}  that there exists $A_i\in\mathbb{M}_n$ with $\|A_i\|=1$ and
  $$\|C_iBC_i^*\|_*=\langle C_iBC_i^*\mid A_i\rangle=\langle B \mid C_i^*A_iC_i\rangle.$$
  Since $\|\cdot\|$ is of $M$-type, $\left\|\sum_{i=1}^{k}C_i^*A_iC_i\right\|\leq \max_{i}\|A_i\|=1$,
  so that
  \begin{align*}
\sum_{i=1}^{k}\|C_iBC_i^*\|_*=\sum_{i=1}^{k}\left\langle B\mid C_i^*A_iC_i\right\rangle&=\sum_{i=1}^{k}{\rm Tr}\left(B^*\left(C_i^*A_iC_i\right)\right)\\
  &={\rm Tr}\left(B^*\left(\sum_{i=1}^{k}C_i^*A_iC_i\right)\right)\\
  &=\left\langle B\left|\ \sum_{i=1}^{k}C_i^*A_iC_i\right.\right\rangle\leq \|B\|_*.
  \end{align*}
  Therefore, $\|\cdot\|_*$ is a $L$-norm.

Suppose conversely that the dual norm $\|\cdot\|_*$ is a $L$-norm. If $A_i\in\mathbb{M}_n$\ $(i=1,\ldots,k)$ and $\sum_{i=1}^{k}C_i^*C_i=I$, then there exits $B\in\mathbb{M}_n$ such that
$$\| B\|_*=1 \ \ \mbox{and} \ \ \left\langle B\left|\ \sum_{i=1}^{k}C_i^*A_iC_i\right.\right\rangle=\left\|\sum_{i=1}^{k}C_i^*A_iC_i\right\|.$$
 We have
 \begin{align*}
  \left\|\sum_{i=1}^{k}C_i^*A_iC_i\right\|& =\left\langle B\left|\ \sum_{i=1}^{k}C_i^*A_iC_i\right.\right\rangle \\ &=\sum_{i=1}^{k}\langle C_iBC_i^* |  A_i\rangle\qquad (\mbox{by \eqref{tr1}})\\
  &\leq \sum_{i=1}^{k}\|A_i\|\cdot \| C_iBC_i^*\|_*\\
  &\leq \max_{1\leq j\leq k}\|A_j\|\sum_{i=1}^{k}\| C_iBC_i^*\|_*\\
  & \leq \max_{1\leq j\leq k}\|A_j\|\cdot\|B\|_*\\
  &=\max_{1\leq j\leq k}\|A_j\|,
 \end{align*}
 where the last inequality follows from the fact that $\|\cdot\|_*$ is a  $L$-norm. This completes the proof.
\end{proof}

We  can   summarize the above discussion   as: \\
If $\|\cdot\|_*$ is the dual norm of a norm $\|\cdot\|$ on $\mathbb{M}_n$ and $\mathscr{B}_1$ and $\mathscr{B}_1^*$ are the unit ball of $\|\cdot\|$ and $\|\cdot\|_*$, respectively, then
$$\mathscr{B}_1 \ \mbox{is $C^*$-convex}\ \Longleftrightarrow\  \|\cdot\|\ \mbox{ is a $M$-norm}\ \Longleftrightarrow\ \|\cdot\|_* \ \mbox{is a $L$-norm}$$
and
$$\mathscr{B}_1^* \ \mbox{is $C^*$-convex}\ \Longleftrightarrow\  \|\cdot\|\ \mbox{ is a $L$-norm}\ \Longleftrightarrow\ \|\cdot\|_* \ \mbox{is a $M$-norm}.$$
\begin{example}
  The trace norm $\|\cdot\|_1$ defined on $\mathbb{M}_n$ by $\|A\|_1={\rm Tr}(|A|)$ is a $L$-norm. Since
  $\|\cdot\|_1$  is the dual norm of the operator norm,  Example \ref{spn} and Lemma \ref{lem2} imply that $\|\cdot\|_1$ is a $L$-norm.
\end{example}

 Let $\|\cdot\|$  be a norm on $\mathbb{M}_n$. There are $M$-norms $\||\cdot\||_{(M1)}$  and $\||\cdot\||_{(M2)}$  and $L$-norms $\||\cdot\||_{(L1)}$  and $\||\cdot\||_{(L2)}$  such that $\||\cdot\||_{(M1)}\leq \|\cdot\|\leq \||\cdot\||_{(M2)}$ and  $\||\cdot\||_{(L1)}\leq \|\cdot\|\leq \||\cdot\||_{(L2)}$.  Since all norms on $\mathbb{M}_n$ are equivalent, there are $\alpha,\beta,\mu,\nu>0$ such that
 $\alpha\|\cdot\|_{\infty}\leq\|\cdot\|\leq \beta\|\cdot\|_{\infty}$ and $\mu\|\cdot\|_{1}\leq\|\cdot\|\leq \nu\|\cdot\|_{1}$.   It is enough to define $\||\cdot\||_{(M1)}:=\alpha\|\cdot\|_{\infty}$, $\||\cdot\||_{(M2)}:=\beta\|\cdot\|_{\infty}$,   $\||\cdot\||_{(L1)}:=\mu\|\cdot\|_{1}$ and $\||\cdot\||_{(L2)}:=\nu\|\cdot\|_{1}$.


Recall that a norm $\|\cdot\|$ on $\mathbb{B}(\mathcal{H})$ is said to be unitarily invariant if
$\|UAV\|=\|A\|$ for all $A\in\mathbb{B}(\mathcal{H})$ and all unitaries  $U,V\in\mathbb{B}(\mathcal{H})$,
while  it is called weakly unitarily invariant if
$\|U^*AU\|=\|A\|$ for every  $ A\in\mathbb{B}(\mathcal{H})$ and every unitary $ U\in\mathbb{B}(\mathcal{H})$. It is easy to see that all $M$-norms and $L$-norms on $\mathbb{B}(\mathcal{H})$ are  weakly unitarily invariant.
Moreover, let $E$ be a projection in $\mathbb{B}(\mathcal{H})$. If $\|\cdot\|$ is a $M$-norm, then  \eqref{c1} gives
 $$\|EAE\|=\|C_1^*A_1C_1+C_2^*A_2C_2\|\leq \max\left\{\|A\|,\|0\|\right\}=\|A\|$$
and if $\|\cdot\|$ is a $L$-norm, then \eqref{cd} implies
 $$\|EAE\|\leq \|C_1AC_1^*\|+\|C_2AC_2^*\|\leq \|A\|.$$
 We will use this fact:
 \begin{lemma}\label{lem5}
   If  $\|\cdot\|$  is a $L$-norm or is a $M$-norm, then $ \|EAE\|\leq \|A\|$ for  every projection $E$ and every $A\in\mathbb{B}(\mathcal{H})$.
\end{lemma}
\begin{proposition}\label{th22}
Let  $\|\cdot\|$ be a norm on $\mathbb{M}_n$.
\begin{enumerate}
  \item If $\|\cdot\|$ is a $M$-norm, then there exists $\alpha>0$ such that $\|A\|=\alpha\|A\|_{\infty}$ for every normal element $A\in\mathbb{M}_n$.\\
      \item If $\|\cdot\|$ is a $L$-norm, then there exists $\alpha>0$ such that $\|A\|=\alpha\|A\|_{1}$ for every normal element $A\in\mathbb{M}_n$.
\end{enumerate}
\end{proposition}
\begin{proof}
Note that if $\|\cdot\|$ is a $M$-norm or is a $L$-norm, then it is weakly unitarily invariant. Moreover, every two rank-one orthoprojections $P$ and $Q$  in $\mathbb{M}_n$ are unitarily equivalent, i.e., there exists a unitary $U$ such that $P=U^*QU$. It follows that $\|P\|=\|U^*QU\|=\|Q\|$.

Now, if $A\in\mathbb{M}_n$ is normal, then by the spectral theorem, there are orthoprojections $E_i$ and  $\lambda_i\in\mathbb{C}$ \ $(i=1,\ldots,n)$ such that
\begin{align}\label{sd}
  A=\sum_{i=1}^{n}\lambda_iE_i;\qquad E_jE_k=\delta_{j,k}E_k,\qquad \sum_{i=1}^{n}E_i=I.
\end{align}
Noticing that $\|E_i\|=\|E_j\|$ for every $i,j\in\{1,\ldots,n\}$, we set $\alpha:=\|E_i\|$. \\
{\rm (1)}\ If $\|\cdot\|$ is a $M$-norm, then
$$\alpha\|A\|_{\infty}=\alpha\cdot\max_{1\leq i\leq n}|\lambda_i|=\max_{1\leq i\leq n}\|\lambda_i E_i\|= \max_{1\leq i\leq n}\|E_iAE_i\|\leq\|A\|,$$
where the last inequality follows from Lemma \ref{lem5}. In addition, Since $\|\cdot\|$ is a $M$-norm, we have
$$\|A\|=\left\|\sum_{i=1}^{n}\lambda_iE_i\right\|=\left\|\sum_{i=1}^{n}E_i(\lambda_iE_i)E_i\right\|\leq\max_{1\leq i\leq n}\|\lambda_iE_i\|=\alpha\|A\|_{\infty}.$$
This completes the proof of (1). \\
{\rm (2)}\  If $\|\cdot\|$ is a $L$-norm, then
$$\|A\|=\left\|\sum_{i=1}^{n}\lambda_iE_i\right\|\leq\sum_{i=1}^{n}\|\lambda_iE_i\|= \alpha\sum_{i=1}^{n}|\lambda_i|=\alpha\|A\|_1.$$
Moreover,
$$\alpha\|A\|_1=\sum_{i=1}^{n}\|\lambda_iE_i\|=\sum_{i=1}^{n}\|E_iAE_i\|\leq \|A\|,$$
since $\|\cdot\|$ is a $L$-norm.
\end{proof}
\begin{corollary}
 Let  $\|\cdot\|$ be a unitarily  invariant  norm on $\mathbb{M}_n$.
  \begin{enumerate}
  \item  If $\|\cdot\|$ is a $M$-norm, then  $\|\cdot\|=\alpha\|\cdot\|_{\infty}$ for some $\alpha>0$. \\
      \item If $\|\cdot\|$ is a $L$-norm, then $\|\cdot\|=\alpha\|\cdot\|_{1}$ for  some $\alpha>0$.
  \end{enumerate}
\end{corollary}
\begin{proof}
  Indeed, if $\|\cdot\|$ is unitarily invariant, then for every $A\in\mathbb{M}_n$ we have $\|A\|=\|\ |A| \ \|$ by the polar decomposition. Proposition  \ref{th22} then can be applied.
\end{proof}
It follows from the proof of Proposition \ref{th22} that if  a norm $\|\cdot\|$ on $\mathbb{M}_n$ is weakly unitarily invariant, then $\|E_1\|=\|E_2\|$ for every two rank-one orthoprojections $E_1$ and $E_2$. If $\|\cdot\|$ is a weakly unitarily invariant norm, then we set $\chi_{\|\cdot\|}:=\|E\|$, where $E$ is any rank-one orthoprojection. Let $\omega(\cdot)$ and $\omega_*(\cdot)$ denote the numerical radius (norm) and its dual norm, respectively.

\begin{proposition}
  Let $\|\cdot\|$  be a weakly unitarily invariant norm with $\chi_{\|\cdot\|}=1$.
\begin{enumerate}
  \item If $\|\cdot\|$  is a $M$-norm, then $\omega(\cdot)\leq \|\cdot\| \leq\|\cdot\|_{\infty}$.
  \item If $\|\cdot\|$  is a $L$-norm, then $\|\cdot\|_1\leq \|\cdot\|\leq\omega_*(\cdot)$.
\end{enumerate}
\end{proposition}

\begin{proof}
  (1) \ First take $A$ with $\|A\|=1$. For any unit vector $\|x\|=1$, put $E:=xx^*$. Since $\|E\|=1$, and $EAE=\langle x,Ax\rangle E$, we have
 \begin{align*}
|\langle x,Ax\rangle|=\|EAE\|&=\|EAE+(I-E)0(I-E)\|\\
& \leq \|A\|=1 \qquad(\mbox{since $\|\cdot\|$ is a $M$-norm})
 \end{align*}
 so that $\omega(A)\leq 1=\|A\|$. \\
 Next, take $B$ with $\|B\|_{\infty}=1$. It is known that $B$ can be written as a convex combination of two unitaries, i.e.,
 $B=\lambda U+(1-\lambda)V$
 for some $\lambda\in[0,1]$ and two unitaries $U$ and $V$. We know that every unitary matrix has norm 1 for any $M$-norm $\|\cdot\|$  with $\chi_{\|\cdot\|}=1$. Therefore
 $$\|B\|\leq \lambda\|U\|+(1-\lambda)\|V\|=1=\|B\|_{\infty}.$$
 (2)\ If $\|\cdot\|$ is a $L$-norm, then $\|\cdot\|_*$ is a $M$-norm. Moreover, $\chi_{\|\cdot\|_*}=1$. Part (1) now implies that $\omega(\cdot)\leq \|\cdot\|_*\leq\|\cdot\|_{\infty}$. Finally, duality of norms reveals that $\|\cdot\|_1\leq \|\cdot\|\leq\omega_*(\cdot)$.
\end{proof}

\section{Minimum norms of $L$-type  and maximum norms of $M$-type }
In this section we are going to show that the class of $L$-norms  enjoys a minimum element.
\begin{theorem}\label{th11}
 The class of $L$-norms on $\mathbb{B}(\mathcal{H})$  grater  than a given norm
possesses a minimum element.
\end{theorem}
\begin{proof}Let $\|\cdot\|$ be a norm on $\mathbb{B}(\mathcal{H})$.
  Define a norm $\|\cdot\|_{(u)}$ by
 \begin{align}\label{u}
  \|A\|_{(u)}:=\sup\ \left\{\sum_{i=1}^{k}\|C_iAC_i^*\|;\ \sum_{i=1}^{k}C_i^*C_i=I\right\}.
 \end{align}
 It is clear that $\|\cdot\|_{(u)}$ becomes a norm greater than $\|\cdot\|$. Let $A\in\mathbb{B}(\mathcal{H})$  and $\sum_{j=1}^{m}D_j^*D_j=I$. By definition \eqref{u}, for every $0<\epsilon<1$ and for every $j=1,\ldots,m$, there exist $C_{ij}$ such that
 \begin{align*}
  \sum_{i=1}^{k}C_{ij}^*C_{ij}=I \quad \ \  \mbox{and} \ \quad  \ (1-\epsilon)\|D_jAD_j^*\|_{(u)}\leq \sum_{i=1}^{k}\|C_{ij}D_jAD_j^*C_{ij}^*\|,
   \end{align*}
so that
 \begin{align}\label{me1}
(1-\epsilon)\sum_{j=1}^{m}\|D_jAD_j^*\|_{(u)}\leq \sum_{j=1}^{m}\sum_{i=1}^{k}\|C_{ij}D_jAD_j^*C_{ij}^*\|.
   \end{align}
Since
$$\sum_{j=1}^{m}\sum_{i=1}^{k}D_j^*C_{ij}^*C_{ij}D_j
=\sum_{j=1}^{m}D_j^*\left(\sum_{i=1}^{k}C_{ij}^*C_{ij}\right)D_j=\sum_{j=1}^{m}D_j^*D_j=I,$$
we have by \eqref{u} that
\begin{align}\label{me2}
\sum_{j=1}^{m}\sum_{i=1}^{k}\|C_{ij}D_jAD_j^*C_{ij}^*\|\leq \|A\|_{(u)}.
\end{align}
Hence $(1-\epsilon)\sum_{j=1}^{m}\|D_jAD_j^*\|_{(u)}\leq\|A\|_{(u)}$ by \eqref{me1} and \eqref{me2}.  Letting $\epsilon\to 0$ we conclude that $\|\cdot\|_{(u)}$ is a $L$-norm.

Finally, let $\||\cdot\||$ be a $L$-norm and greater than $\|\cdot\|$. If $A\in\mathbb{B}(\mathcal{H})$ and $\sum_{i=1}^{k}C_i^*C_i=I$, then
$$\sum_{i=1}^{k}\|C_iAC_i^*\|\leq\sum_{i=1}^{k}\||C_iAC_i^*\|| \leq \||A\||.$$
Now definition \eqref{u} implies that $\|A\|_{(u)}\leq  \||A\||$. Therefore, $\|\cdot\|_{(u)}$ is the minimum element.
\end{proof}

\begin{corollary}\label{co11}
  The class of $M$-norms on $\mathbb{M}_n$  less  than a given norm $\|\cdot\|$
possesses a minimum element.
\end{corollary}
\begin{proof}
 Set
 $$\mathscr{P}=\left\{\||\cdot\||;\ \ \||\cdot\||\ \mbox{is a $M$-norm and }\ \  \||\cdot\||\leq \|\cdot\|\right\}$$
 and
 $$\mathscr{P}_*=\left\{\||\cdot\||_*;\ \ \||\cdot\||\in\mathscr{P}\right\}=\left\{\||\cdot\||_*;\ \ \||\cdot\||_*\ \mbox{is a $L$-norm and }\ \  \||\cdot\||_*\geq \|\cdot\|_*\right\}.$$
 Theorem \ref{th11} implies that $\mathscr{P}_*$ has a minimum element, say $\|\cdot\|_{*(u)}$. Duality of norms and Lemma \ref{lem2} ensure that $\|\cdot\|_{(u)}$ is a $M$-norm. Moreover,  $\|\cdot\|_{(u)}$ is the maximum element of $\mathscr{P}$ by definition.
\end{proof}
Consider the operator norm on $\mathbb{M}_n$. Theorem \ref{th11} guarantees that the class of $L$-norms grater than the operator norm  has  a minimum element. In the next result we show that the trace norm $\|\cdot\|_1$  is the minimum one along the class of $L$-norms   bounded below  by the operator norm.

\begin{theorem}
The trace norm $\|\cdot\|_1$ is the minimum element in the class of $L$-norms greater than the operator norm $\|\cdot\|_\infty$, i.e.,
 \begin{align*}
  \|A\|_1=\sup\ \left\{\sum_{i=1}^{k}\|C_iAC_i^*\|_{\infty};\ \  \sum_{i=1}^{k}C_i^*C_i=I\right\}.
 \end{align*}
\end{theorem}
\begin{proof}
For any matrix $A\in\mathbb{M}_n$ and every orthonormal system $x_j$\ $(j=1,\ldots,n)$ we can write
\begin{align*}
  \sum_{j=1}^{n}|\langle x_j,Ax_j\rangle|\leq\sum_{j=1}^{n}\|E_jAE_j\|_{\infty}\leq\sum_{j=1}^{n}\|E_jAE_j\|_1\leq \|A\|_1,
\end{align*}
 in which $E_j$ is the rank-one  orthoprojection to $\mathbb{C}x_j$.  Moreover, assume that $A=U|A|$ be the polar decomposition of $A$ with unitary $U$. Then
 \begin{align}\label{e10}
   \|A\|_1=\mbox{Tr}(|A|)=\mbox{Tr}(U|A|U^*)=\mbox{Tr}(AU^*).
 \end{align}
 Let
 \begin{align}\label{e11}
  U^*=\sum_{j=1}^{n}e^{i\theta_j}x_jx_j^*
 \end{align}
 be the spectral decomposition of unitary $U^*$, where $x_j$'s are the (column) unit eigenvectors ( so that $x_j^*$'s are row vectors and $x_jx_j^*=E_j$ is the rank-one projection onto the subspace $\mathbb{C}x_j$). Now, it follows from \eqref{e10} and \eqref{e11} that
  {\small\begin{align}\label{q211}
   \|A\|_1=\mbox{Tr}(AU^*)=\sum_{j=1}^{n}e^{i\theta_j}\mbox{Tr}(AE_j)=\sum_{j=1}^{n}e^{i\theta_j}\langle x_j,Ax_j\rangle
   \leq\sum_{j=1}^{n}|\langle x_j,Ax_j\rangle|\leq\sum_{j=1}^{n}\|E_jAE_j\|_{\infty}.
 \end{align}}
This implies that
 \begin{align*}
   \|A\|_1&\leq \sup\ \left\{\sum_{i=1}^{n}|\langle x_j,Ax_j\rangle|;\ \ \mbox{$\{x_j\}$ is an orthonormal system in}\  \mathbb{C}^n \right\}\quad(\mbox{by \eqref{q211}})\\
  & \leq \sup \left\{\sum_{i=1}^{n}\|E_iAE_i^*\|_{\infty};\ \ \mbox{$E_i$'s are rank-one orthoprojections},\ \  \sum_{i=1}^{n}E_i=I\right\}\quad(\mbox{by \eqref{q211}})\\
  &\leq \sup\ \left\{\sum_{i=1}^{k}\|C_iAC_i^*\|_{\infty};\ \  \sum_{i=1}^{k}C_i^*C_i=I\right\}\\
  &\leq \sup\ \left\{\sum_{i=1}^{k}\|C_iAC_i^*\|_1;\ \  \sum_{i=1}^{k}C_i^*C_i=I\right\}\\
  &\leq \|A\|_1 \qquad \quad(\mbox{Since $\|\cdot\|_1$ is a $L$-norm }).
  \end{align*}
    Therefore, they are all equal and the proof is complete.
\end{proof}

We give another example for Theorem \ref{th11} in the next theorem.
\begin{theorem}\label{th33}
  The trace norm $\|\cdot\|_1$ is the minimum element in the class of $L$-norms greater than the
numerical radius norm $\omega(\cdot)$.
\end{theorem}
\begin{proof}
  Assume that $\|\cdot\|$ is a $L$-norm with $\omega(\cdot)\leq\|\cdot\|$. For every $A\in\mathbb{M}_n$, there exits an orthonormal system of vectors $\{x_j\}$ such that
  \begin{align*}
\|A\|_1=\sum_{j=1}^{n}|\langle x_j,Ax_j\rangle|&\leq\sum_{j=1}^{n}\omega(E_jAE_j)\\
&\leq\sum_{j=1}^{n}\|E_jAE_j\|\leq \|A\|,
  \end{align*}
  where $E_j=x_jx_j^*$. This ensures that  $\|\cdot\|_1$ is the minimum $L$-norm greater than the
numerical radius norm.
\end{proof}
\begin{corollary}
  The operator norm $\|\cdot\|_{\infty}$ is the maximum element in the class of  $M$-norms smaller than
$\omega_*(\cdot)$, the dual norm of $\omega(\cdot)$.
\end{corollary}
\begin{proof}
  Apply Corollary \ref{co11} and Theorem \ref{th33}.
\end{proof}
A natural question is that does the class of $L$-norms $\||\cdot\||$  which are less than a norm $\|\cdot\|$  contain a maximum element? or does the class of $M$-norms $\||\cdot\||$ which are  greater than a norm $\|\cdot\|$ contain a minimum element?

We answer these questions in some special cases below. First we need the following simple lemma.
\begin{lemma}\label{1-norm}
If $P$ and $Q$ are orthoprojections in $\mathbb{B}(\mathcal{H})$, then
$$\|PXP+QYQ\|_1=\|PXP\|_1+\|QYQ\|_1$$
for every $X,Y\in\mathbb{L}^1(\mathcal{H})$.
  \end{lemma}
\begin{proof}
  First note that $\mathrm{Tr}(U^*AU)=\mathrm{Tr}(A)$ for every $A\in\mathbb{L}^1(\mathcal{H})$ and every unitary $U$. If $P$ and $Q$ are orthoprojections in $\mathbb{B}(\mathcal{H})$, then $U=\left[\begin{array}{cc}P& Q\\
  Q &P\end{array}\right]$ is a unitary in $\mathbb{B}(\mathcal{H}\oplus\mathcal{H})$. If $A,B\in\mathbb{L}^1(\mathcal{H})$, then
  \begin{align}\label{q1}
  \mathrm{Tr}(PAP+QBQ)+\mathrm{Tr}(QAQ+PBP)&=    \mathrm{Tr}\left(\left[\begin{array}{cc}PAP+QBQ& PAQ+QBP\\
  QAP+PBQ &QAQ+PBP\end{array}\right]\right)\nonumber\\
  &=\mathrm{Tr}\left(\left[\begin{array}{cc}P& Q\\
  Q &P\end{array}\right]\left[\begin{array}{cc}A& 0\\
  0 &B\end{array}\right]\left[\begin{array}{cc}P& Q\\
  Q &P\end{array}\right]\right)\nonumber\\
  &=\mathrm{Tr}\left(\left[\begin{array}{cc}A& 0\\
  0 &B\end{array}\right]\right)=\mathrm{Tr}(A)+\mathrm{Tr}(B).
  \end{align}
  Put $A=(PX^*PXP)^{\frac{1}{2}}=|PXP|$ and $B=(QY^*QYQ)^{\frac{1}{2}}=|QYQ|$. Then
  \begin{align}\label{q2}
\mathrm{Tr}(PAP+QBQ)&=\mathrm{Tr}\left(P\left(PX^*PXP\right)^{\frac{1}{2}}P+Q(QY^*QYQ)^{\frac{1}{2}}Q\right)\nonumber\\
&\leq \mathrm{Tr}\left(PX^*PXP+QY^*QYQ\right)^{\frac{1}{2}}\quad(\mbox{by the Jensen inequality})\nonumber\\
&=\mathrm{Tr}\left(((PXP+QYQ)^*(PXP+QYQ))^{\frac{1}{2}}\right)\nonumber\\
&\qquad\qquad\qquad\qquad\qquad\qquad\quad(\mbox{by orthogonality of $P$ and $Q$})\nonumber\\
&=\mathrm{Tr}(|PXP+QYQ|)=\|PXP+QYQ\|_1.
  \end{align}
  Moreover,
    \begin{align}\label{q3}
\mathrm{Tr}(QAQ+PBP)&=\mathrm{Tr}\left(Q\left(PX^*PXP\right)^{\frac{1}{2}}Q+P(QY^*QYQ)^{\frac{1}{2}}P\right)\nonumber\\
&\leq \mathrm{Tr}\left(QPX^*PXPQ+PQY^*QYQP\right)^{\frac{1}{2}}\quad(\mbox{by the Jensen inequality})\nonumber\\
&=0.
  \end{align}

  It follows from \eqref{q1} and \eqref{q2} and \eqref{q3} that
  \begin{align*}
\|PXP\|_1+\|QYQ\|_1=\mathrm{Tr}(|PXP|)+\mathrm{Tr}(|QYQ|)=\mathrm{Tr}(A)+\mathrm{Tr}(B)&\leq\|PXP+QYQ\|_1.
  \end{align*}
  This completes the proof.
\end{proof}
\begin{theorem}\label{th44}
   The norm $n\cdot \omega(\cdot)$ is the minimum element  in the class of $M$-norm greater than the trace norm $\|\cdot\|_1$.
\end{theorem}
\begin{proof}
Note first that the unit ball of the numerical radius norm $\omega(\cdot)$ is $C^*$-convex \cite{Loebl} and so $\omega(\cdot)$ is a  $M$-norm.
 For every $A\in\mathbb{M}_n$ there exists an orthonormal system of vectors $\{x_j\}_{j=1}^{n}$  of $\mathbb{C}^n$ such that $\|A\|_1=\sum_{j=1}^{n}|\langle x_j,Ax_j\rangle|\leq\sum_{j=1}^{n}\omega(A)= n\ \omega(A) $. So, $n\ \omega(\cdot)$ is a  $M$-norm greater than the trace norm $\|\cdot\|_1$.

 Now let $\|\cdot\|$ be a  $M$-norm greater than the trace norm. For any $A\in\mathbb{M}_n$, there is $x\in\mathbb{C}^n$ such that $\|x\|=1$ and $\omega(A)= |\langle x,Ax\rangle|$. Starting with $x_1 := x$, construct a coordinate orthonormal system $\{x_j\}$ in $\mathbb{C}^n$. Assume that $U_j$ is the unitary such that $U_jx_j=x$. Let $E_j:=x_jx_j^*$ be the rank-one orthoprojection to $\mathbb{C}x_j$.  Note that
 \begin{align}\label{w}
   \|E_jU_j^*AU_jE_j\|_1&=\|\langle x_j,(U_j^*AU_j)x_j\rangle E_j\|_1\nonumber\\
   &= |\langle x_j,(U_j^*AU_j)x_j\rangle|=|\langle U_jx_j,AU_jx_j\rangle|=|\langle x,Ax\rangle|
   =\omega(A).
    \end{align}
    Since $\|\cdot\|$ is weakly unitary invariant and $\|\cdot\|_1\leq\|\cdot\|$, we have
     \begin{align*}
     n\cdot \omega(A)=n\ |\langle x,Ax\rangle|&=\sum_{j=1}^{n}\|E_jU_j^*AU_jE_j\|_1\\
     &=\left\|\sum_{j=1}^{n}E_jU_j^*AU_jE_j\right\|_1\qquad (\mbox{by Lemma \ref{1-norm}})\\
     &\leq \left\|\sum_{j=1}^{n}E_jU_j^*AU_jE_j\right\|\leq \max_{1\leq j\leq n}\|U_j^*AU_j\|=\|A\|.
    \end{align*}
This proves that $ n\cdot \omega(\cdot)$ is the minimum element.
\end{proof}
\begin{corollary}
The norm $\frac{\omega_*(\cdot)}{n}$  is the maximum element in the class of  $L$-norms smaller than the operator  norm $\|\cdot\|_{\infty}$.
\end{corollary}
\begin{proof}
  This follows from  Theorem \ref{th44} via duality of corresponding norms.
\end{proof}

\section{Construction of $M$-norms and $L$-norms}
If $f:(0,\infty)\to(0,\infty)$ is   an  operator  concave function, the Jensen's operator inequality \cite{HP} implies that
\begin{align}
  f\left(\sum_{j=1}^{k}C_j^*A_jC_j\right)\geq\sum_{j=1}^{k}C_j^*f(A_j)C_j
\end{align}
for all $A_j\geq0$ and $C_j\in\mathbb{B}(\mathcal{H})$ with $\sum_{j=1}^{k}C_j^*C_j=I$.

\begin{theorem}\label{mt}
  Assume that $\varphi(t),\psi(t)>0$ are  matrix concave functions on $(0,1)$. The set
  $$\mathcal{U}_{\varphi,\psi}=\left\{X;\quad \left[
\begin{array}{cc}
 \varphi(A) &  X\\
   X^* & \psi(A)
\end{array}\right]\geq0, \quad \mbox{for some}\  \ 0\leq A\leq I\right\}$$
is $C^*$-convex.
\end{theorem}
\begin{proof}
  Let  $C_i\in\mathbb{B}(\mathcal{H})$ with $\sum_{i=1}^{k}C_i^*C_i=I$. If  $X_1,\ldots,X_k\in\mathcal{U}_{\varphi,\psi}$, then there exist $A_1,\ldots,A_n$ with $0\leq A_i\leq I$ $(i=1,\ldots,k)$ such that $\left[
\begin{array}{cc}
 \varphi(A_i) &  X_i\\
   X_i^* & \psi(A_i)
\end{array}\right]\geq0$,
for all $i=1,\ldots,k$. It follows that
\begin{align}\label{c2}
\varphi(A_i) \geq X_i \psi(A_i)^{-1}X_i^*, \qquad i=1,\ldots,k.
\end{align}
 Therefore
 \begin{align}\label{c3}
 \varphi\left(\sum_{i=1}^{k}C_i^*A_iC_i\right)&\geq \sum_{i=1}^{k}C_i^*\varphi(A_i)C_i \qquad\qquad (\mbox{by matrix concavity of $\varphi$})\nonumber\\
 &\geq \sum_{i=1}^{k}C_i^*X_i \psi(A_i)^{-1}X_i^*C_i\qquad (\mbox{by \eqref{c2}}).
\end{align}
It is easy to see that   the function $g(X,Y)=X\psi(Y)^{-1}X^*:\mathbb{B}(\mathcal{H})\times\mathbb{B}(\mathcal{H})^+\to\mathbb{B}(\mathcal{H})^+$ is jointly convex.  Hence
\begin{align}\label{c4}
  \left(\sum_{i=1}^{k}C_i^*X_iC_i\right)\psi\left(\sum_{i=1}^{k}C_i^*A_iC_i\right)^{-1}
  \left(\sum_{i=1}^{k}C_i^*X_i^*C_i\right)&=g\left(\sum_{i=1}^{k}C_i^*X_iC_i,\sum_{i=1}^{k}C_i^*A_iC_i\right)\nonumber\\
  &\leq  \sum_{i=1}^{k}C_i^*g(X_i,A_i)C_i\nonumber\\
  &=  \sum_{i=1}^{k}C_i^*X_i\psi(A_i)^{-1}X_i^*C_i.
\end{align}
Combining \eqref{c3} and \eqref{c4} we get
$$\varphi\left(\sum_{i=1}^{k}C_i^*A_iC_i\right)\geq \left(\sum_{i=1}^{k}C_i^*X_iC_i\right)\psi\left(\sum_{i=1}^{k}C_i^*A_iC_i\right)^{-1}
  \left(\sum_{i=1}^{k}C_i^*X_i^*C_i\right).$$
  It follows that
  $$\left[
\begin{array}{cc}
 \varphi(\sum_{i=1}^{k}C_i^*A_iC_i) &  \sum_{i=1}^{k}C_i^*X_i^*C_i\\
  \left(\sum_{i=1}^{k}C_i^*X_iC_i\right)^* & \psi(\sum_{i=1}^{k}C_i^*A_iC_i)
\end{array}\right]\geq0.$$
Noting that $0\leq \sum_{i=1}^{k}C_i^*A_iC_i\leq \sum_{i=1}^{k}C_i^*C_i=I$, we then conclude that \\ $\sum_{i=1}^{k}C_i^*X_iC_i\in~\mathcal{U}_{\varphi,\psi}$.
\end{proof}

Theorem \ref{mt} provides  a way of constructing  $M$-norms.
If we find any norm such that its unit ball coincide with $\mathcal{U}_{\varphi,\psi}$ for some proper matrix concave functions $\varphi$ and $\psi$, then that norm would be a $M$-norm by Lemma \ref{lem1}. An obvious example is the spectral norm $\|\cdot\|_\infty$ whose unit ball coincide with  $\mathcal{U}_{\varphi,\psi}$, where $\varphi(t)=\psi(t)=1$, since
 $$\|A\|_{\infty}\leq 1 \ \Longleftrightarrow \ \left[
\begin{array}{cc}
 I &  A\\
   A^* & I
\end{array}\right]\geq0.$$

As another example consider the  numerical radius norm $\omega(\cdot)$. It is known that \cite{Bhatia} that for
 $X\in\mathbb{M}_n$,  $\omega(X)\leq 1$ if and only if there exists a Hermitian
matrix $H$ with $-I\leq H\leq I$ such that
$\left[\begin{array}{cc}
I+H & X\\
X^* & I-H
\end{array}\right]$
is positive.

With  $A=\frac{I+H}{2}$ this concludes that    $0\leq A\leq I$ and
\begin{align}\label{q112}
\omega(X)\leq 1 \ \ \Longleftrightarrow \ \ \left[\begin{array}{cc}
2A & X\\
X^* & 2(I-A)
\end{array}\right]\geq 0, \ \mbox{for some}\  0\leq A\leq I.
\end{align}
If $\varphi(t)=2t$ and $\psi(t)=2(1-t)$, then $\varphi$ and $\psi$ are  matrix concave functions on $(0,1)$. Theorem \ref{mt} then ensures the $C^*$-convexity of $\mathcal{U}_{2t,2(1-t)}$. In addition, $\mathcal{U}_{2t,2(1-t)}=\left\{X;\quad \omega(X)\leq 1\right\}$  by \eqref{q112}. This implies that the numerical radius norm $\omega(\cdot)$ is a $M$-norm.

\bigskip


\end{document}